\documentclass[12pt, reqno]{amsart}

\usepackage{amssymb, amsmath, amsthm}
\usepackage[backref, colorlinks = true, linkcolor = blue, citecolor = Green]{hyperref}
\usepackage[alphabetic,backrefs,lite]{amsrefs}
\usepackage{amscd}   
\usepackage{fullpage}
\usepackage[all]{xy} 

\DeclareFontEncoding{OT2}{}{} 

\usepackage[usenames,dvipsnames]{color}

\newtheorem{lemma}{Lemma}[section]
\newtheorem{theorem}[lemma]{Theorem}

\newtheorem{prop}[lemma]{Proposition}
\newtheorem{cor}[lemma]{Corollary}

\newtheorem{claim*}{Claim}

\theoremstyle{remark}
\newtheorem{remark}[lemma]{Remark}

\newcommand{\A}{{\mathbb A}}
\newcommand{\G}{{\mathbb G}}

\newcommand{\PP}{{\mathbb P}}

\newcommand{\Q}{{\mathbb Q}}

\newcommand{\Z}{{\mathbb Z}}

\newcommand{\Adeles}{{\mathbb A}}
\newcommand{\kk}{{\mathbf k}}

\usepackage[mathscr]{euscript}

\DeclareMathOperator{\HH}{H}

\DeclareMathOperator{\inv}{inv}

\DeclareMathOperator{\im}{im}

\DeclareMathOperator{\Cor}{Cor}
\DeclareMathOperator{\Res}{Res}

\DeclareMathOperator{\Br}{Br}

\DeclareMathOperator{\Pic}{Pic}

\DeclareMathOperator{\Spec}{Spec}

\DeclareMathOperator{\ev}{ev}

\DeclareMathOperator{\et}{et}

\DeclareMathOperator{\CH}{CH}

\newcommand{\dualP}[1]{\check{\PP}^{#1}}
\numberwithin{equation}{section}
\numberwithin{table}{section}

\title{Persistence of the Brauer-Manin obstruction on cubic surfaces}

\author{Carlos Rivera}
\address{University of Washington, Department of Mathematics, Box 354350, Seattle, WA 98195,~USA}
\email{caariv@uw.edu}

\author{Bianca Viray}

\address{University of Washington, Department of Mathematics, Box 354350, Seattle, WA 98195,~USA}
\email{bviray@uw.edu}
\urladdr{http://math.washington.edu/\~{}bviray}

\date{}

\begin{document}

	\begin{abstract}
	    Let $X$ be a cubic surface over a global field $k$.  We prove that a Brauer-Manin obstruction to the existence of $k$-points on $X$  will persist over every extension $L/k$ with degree relatively prime to $3$.  In other words, a cubic surface has nonempty Brauer set over $k$ if and only if it has nonempty Brauer set over some extension $L/k$ with $3\nmid[L:k]$. Therefore, the conjecture of Colliot-Th\'el\`ene and Sansuc on the sufficiency of the Brauer-Manin obstruction for cubic surfaces implies that $X$ has a $k$-rational point if and only if $X$ has a $0$-cycle of degree $1$. This latter statement is a special case of a conjecture of Cassels and Swinnerton-Dyer. 
	\end{abstract}

	\maketitle

\section{Introduction}

    Let $Y$ be a smooth cubic hypersurface over a field $k$. Cassels and Swinnerton-Dyer have conjectured that $Y$ has a rational point if and only if $Y$ has a $0$-cycle of degree $1$ or, equivalently, that $Y$ has a $k$-rational point if and only if $Y$ has an $L$-rational point for a finite extension $L/k$ whose degree is relatively prime to $3$~\cite{Coray-Cubic}.  Note that if $Y$ is a curve, then this conjecture follows from the Riemann-Roch Theorem. 
    
    Coray took up this question for higher-dimensional hypersurfaces and proved (among other results) that this conjecture holds over local fields~\cite{Coray-Cubic}*{Thm. 4.7}. Thus, the Cassels--Swinnerton-Dyer conjecture holds over global fields whenever $Y$ satisfies the local-to-global principle.  Conjecturally, smooth cubic hypersurfaces of dimension at least $3$ satisfy the local-to-global principle~\cite{CTConj}.
    
    Cubic \emph{surfaces} can fail the local-to-global principle~\cite{SwinnertonDyer62}, and we have a conjectural understanding of all such failures.  Indeed, Colliot-Th\'el\`ene and Sansuc have conjectured that the Brauer-Manin obstruction is the only obstruction to the local-to-global principle.  That is, if $X$ is a cubic surface over a global field $k$, then a nonempty Brauer set $Y(\A_k)^{\Br}\subset Y(\A_k)$ should imply the existence of a $k$-rational point. We prove that this conjecture of Colliot-Th\'el\`ene and Sansuc implies the conjecture of Cassels and Swinnerton-Dyer over global fields.  More precisely, our main theorem is the following.
    
\begin{theorem}\label{thm:Main}
    Let $X$ be a smooth cubic surface over a global field $k$. If $L/k$ is an extension with degree coprime to $3$, then
    \[
        X(\A_L)^{\Br} = \emptyset \Leftrightarrow X(\A_k)^{\Br} = \emptyset.
    \]
\end{theorem}
\begin{cor}
    Let $X$ be a smooth cubic surface over a global field $k$.  Assume that the Brauer-Manin obstruction is the only obstruction to the local-to-global principle for cubic surfaces over global fields.  Then $X$ has a $k$-rational point if and only if $X$ has a $0$-cycle of degree $1$. \hfill\qed
\end{cor}
The key insight in the proof is that we need only understand $n\CH_0(X)$ for some $n$ coprime to $3$ to compute the Brauer-Manin obstruction over extensions. (This reduction, which will be explained in detail in the proof, is due to the bilinearity of the Brauer pairing and the fact that Brauer elements must be of order $3$ to obstruct the local-to-global principle~\cite{SD-BrauerCubic}*{Cor. 1}).  We extend a result of Colliot-Th\'el\`ene~\cite{CT-CorayVariations}*{Thm. 3.3e} to obtain this desired understanding of $2\CH_0(X)$. 

\subsection*{Conventions and notation} For a smooth proper variety \(X\) over a field \(k\), we write \(\Br X\) for the cohomological Brauer group \(\HH^2_{\et}(X, \G_m)\) and write \(\CH_0(X)\) for the Chow group of \(0\)-cycles modulo rational equivalence. We also denote \(\Br \Spec k\) by \(\Br k\).  

Given a field extension \(F/k\) and an element \(\alpha\in \Br X\), we write \(\ev_{\alpha}\colon X(F)\to \Br F\) for the map that sends a point \(P\in X(F)\) to the pullback \(P^*\alpha\) of \(\alpha\) along \(P\).  This evaluation map respects rational equivalence, so we may extend this definition to obtain a pairing
\[
  \Br X \times \CH_0(X) \to \Br k, \quad \langle \alpha, \sum_i n_iP_i \rangle:= \sum_i n_i\Cor_{\kk(P_i)/k}\left(\ev_{\alpha_{\kk(P_i)}}(P_i)\right).  
\]
Note that given a degree \(d\) point \(P\in X\), we may consider this either as an element of \(\CH_0(X)\) or as an element of \(X(\kk(P))\).  In either case, we may pair with a Brauer class, but the pairings are different.  As an element of \(\CH_0(X)\), the pairing \(\langle\alpha, P\rangle\) gives an element of \(\Br k\), whereas as an element of \(X(\kk(P))\), the pairing \(\ev_{\alpha_{\kk(P)}}(P)\) gives an element of \(\Br \kk(P)\).  However, we do have the relation \(\Cor_{\kk(P)/k}(\ev_{\alpha_{\kk(P)}}(P)) = \langle \alpha, P\rangle\).  To avoid confusion, we will use \(\langle \alpha, -\rangle\) to denote the pairing on \(\CH_0(X)\) and \(\ev_{\alpha_L}\) to denote the pairing on \(X(L)\), for \(L\) an extension of \(k\).

If \(k\) is a global field and \(v\) is a place, then we define the invariant maps \(\inv_v\colon \Br k_v \to \Q/\Z\) compatibly so that we have an exact sequence
\[
    0\to \Br k \to \oplus_{v}\Br k_v \xrightarrow{\sum_v\inv_v} \Q/\Z\to 0;
    \]
see~\cite{CTS-Brauer}*{Def. 13.1.7 and Rmk. 13.1.12} for more details.  Recall that \(\inv_v\) is an isomorphism for nonarchimedean \(v\).

\section*{Acknowledgements}
    The second author was supported by NSF grants \#1553459 and \#2101434, and by Simons Foundation Fellowship \#682109.  She thanks the NSF and the Simons Foundation for the support.  The authors thank Anton Geraschenko whose answer on a \texttt{MathOverflow} post~\cite{MO} provided a key reference allowing the extension of the results to include fields of positive characteristic, and thank Karl Schwede and Brendan Creutz for helpful comments.  The authors also thank the anonymous referee for their thorough reading and remarks.

\section{Persistence of constant evaluation over local fields}

    In~\cite{CT-CorayVariations}, Colliot-Th\'el\`ene revisits the aforementioned work of Coray~\cite{Coray-Cubic} and develops a more flexible version of Coray's original methods.  In doing so, Colliot-Th\'el\`ene obtains stronger results on the Chow group of $0$-cycles for cubic surfaces~\cite{CT-CorayVariations}*{Thm. 3.3} (as well as proving analogues of Coray's results for other varieties).  
    
    In this section, we follow Colliot-Th\'el\`ene's proof to give refined information on $2\CH_0 X$ (Lemma~\ref{lem:CTextension}) which we then use to prove that, over local fields, constant Brauer evaluation persists over any extension (Proposition~\ref{prop:ConstantPersistence}).

    \begin{lemma}[Extension of~\cite{CT-CorayVariations}*{Thm. 3.3e}]\label{lem:CTextension}
        Let $k$ be an infinite field  and let $X\subset \PP^3_k$ be a smooth cubic surface. If $X(k) \neq \emptyset$, then the group $2\CH_0(X)$ is generated by classes of $k$-rational points.
    \end{lemma}
    \begin{remark}\label{rmk:Bertini}
        The result of Colliot-Th\'el\`ene that we extend (\cite{CT-CorayVariations}*{Thm. 3.3e}) is stated for fields of characteristic $0$.  The assumption on characteristic is used in~\cite{CT-CorayVariations}*{Proof of Thm. 2.9}, which is a refined Bertini result (see~\cite{Jouanolou}).  However, for~\cite{CT-CorayVariations}*{Thm. 3.3e} and this extension, we need only apply this refined Bertini theorem to embeddings of degree $2$ and $3$ del Pezzo surfaces given by a (large enough) multiple of the anticanonical bundle and to the degree $2$ map $S\to \PP^3$ given by $|-2K_S|$ for $S$ a degree $1$ del Pezzo surface.  The embeddings are unramified and so a generic hyperplane section is smooth, even in positive characteristic.  Furthermore, away from characteristic $2$, the degree $2$ map is of finite type and residually separable so by~\cite{Spreafico}*{Section 4}, a generic hyperplane section is smooth.  In characteristic $2$, we prove directly that a generic hyperplane section is smooth (see Proposition~\ref{char2Bertini}).
    \end{remark}
    \begin{proof}
        By \cite{CT-CorayVariations}*{Thm. 3.3e} we know that $\CH_0(X)$ is generated by classes of $k$-rational points and closed points of degree $3$. Moreover, the standard technique of considering a line through a degree $2$ point shows that every degree $2$ point is already a sum of two $k$-rational points in $\CH_0(X)$. Hence, to prove the lemma it is enough to show that if $Q$ is a degree $3$ closed point of $X$, then $2Q$ is rationally equivalent to a linear combination of points of degree $1$ or $2$.
    
        Following \cite{CT-CorayVariations}*{Proof of Thm. 3.3} we let $R,S$ be general $k$-rational points in $X$, take the blow up $p\colon Y \rightarrow X$ of $X$ at $R$ and $S$, and consider the line bundle 
        \[
            -2K_Y=p^*(-2K_X)-2E_R-2E_S,
        \]
        where $E_R$ and $E_S$ are the exceptional divisors above $R$ and $S$. Since $Y$ is a del Pezzo surface of degree $1$, the linear system $|-2K_Y|$ defines a degree $2$ map $f\colon Y \rightarrow \PP^3$ whose image is a quadric cone.  As $p^*(Q) \in Y$ is a closed point of degree $3$ in $Y$, and the image of $f$ is two dimensional and it generates $\PP^3$, we may apply \cite{CT-CorayVariations}*{Thm. 2.9(b)} to find a smooth geometrically integral $k$-rational curve $\Gamma$ in $|-2K_Y|$ and an effective degree \(3\) divisor $z\subset\Gamma$ that is rationally equivalent to $p^*(Q)$.
        
        By adjunction, we see that 
        \[
            g(\Gamma)= 1+\frac{\Gamma.(\Gamma+K_Y)}{2}=2.
        \]
        Since $\Gamma. E_R=2$,  $w:=\Gamma \cap E_R$ is an effective degree $2$ divisor in $\Gamma$. Applying Riemann-Roch to the degree $2$ divisor $2z-2w$ on the genus $2$ curve $\Gamma$, we find an effective degree $2$ divisor $w'\subset \Gamma$ such that $w'=2z-2w \in \Pic\Gamma$. Since $z\equiv p^*Q\in \CH_0Y$, this shows $2p^*Q\sim w' + 2w$, and so, in particular, $2Q$ is rationally equivalent on \(X\) to a sum of degree $1$ and $2$ points.
    \end{proof}

\begin{prop}\label{prop:ConstantPersistence}
    Let $X$ be a cubic surface over a local field $k$ with $X(k)\neq \emptyset$ and let $\alpha\in \Br X[3]$.  If $\ev_{\alpha}\colon X(k) \to \Br k$ is constant, then for all finite extensions $L/k$, $\ev_{\alpha_L}\colon X(L) \to \Br L$ is constant with image equal to $\Res_{L/k}(\im \ev_{\alpha})$.  In particular, if \(\inv_k\circ\ev_{\alpha}\) has image \(c_{\alpha}\in \Q/\Z\), then \(\inv_L\circ\ev_{\alpha_L}\) has image \([L:k]c_{\alpha}\).
\end{prop}
\begin{proof}
    The last statement follows from the first since \(\inv_L\circ \Res_{L/k} = [L:k]\circ\inv_k\)~\cite{Poonen-RationalPoints}*{Theorem 1.5.34 (ii)}. Also note that if $P\in X(L)$ is contained in $X(F)$ for a subextension $k\subset F \subset L$, then $\ev_{\alpha_L}(P) = \Res_{L/F}(\ev_{\alpha_F}(P))$.  Thus, it suffices to prove constancy on points $P\in X(L)$ that are not defined over any proper subfield of $L$, i.e., those points that define $0$-cycles over $k$ of degree $[L:k]$.  In addition, since $\Cor_{L/k}$ is an isomorphism, it suffices to prove that $(\Cor_{L/k}\circ \ev_{\alpha_L})(P) = [L:k](\ev_{\alpha}(Q))$ for any point \(Q\in X(k)\). 
    
    Let \(P\in X\) be a closed point of degree \(d\) and let \(L= \kk(P)\).  By definition of the Brauer-Manin pairing, \(\Cor_{L/k}\ev_{\alpha_L}(P)\) is equal to the pairing \(\langle\alpha, P\rangle\), where \(P\) is considered as a \(0\)-cycle. Since \(\ev_{\alpha}\) is constant on \(X(k)\), by Lemma~\ref{lem:CTextension}, \(\langle\alpha, -\rangle \) is constant on each degree \(d\) part of \(2\CH_0(X)\).  Furthermore, for any \(Q\in X(k)\) and any degree \(d\) \(0\)-cycle \(D\in 2\CH_0(X)\), we have \(\langle\alpha, D\rangle = d\langle \alpha, Q\rangle\).  Combining these facts with the fact that \(\alpha\) is \(3\)-torsion, we may compute
    \[
        \langle\alpha, P\rangle = \langle4\alpha, P\rangle = 2\langle\alpha, 2P\rangle = 4[L:k]\langle\alpha, Q\rangle = [L:k]\langle\alpha, Q\rangle = \Res_{L/k}(\ev_{\alpha}(Q)).\qedhere
    \] 
\end{proof}

\section{Proof of Theorem~\ref{thm:Main}}

    \begin{lemma}[\cite{CTP-CubicSurfaces}*{Proof of Lemma 3.4} ] \label{lem:ImageEvIsCoset}
        Let $X$ be a smooth cubic surface in $\PP^3$ over a local field  $k$ such that $X(k) \neq \emptyset$. Then for each $\alpha \in {\Br}(X)$ the image of the evaluation map $  \ev_{\alpha}\colon X(k) \rightarrow \Br(k)$ is a group coset.
    \end{lemma}
    \begin{remark}
        The proof of Lemma 3.4 in~\cite{CTP-CubicSurfaces} can be applied verbatim to prove this lemma.  However, since the statement~\cite{CTP-CubicSurfaces}*{Lemma 3.4} differs from that of Lemma~\ref{lem:ImageEvIsCoset}, we repeat the proof for the readers' convenience.
    \end{remark}
    \begin{proof} 
        Let $S = \im \left(\ev_{\alpha}\colon X(k) \to \Q/\Z\right)$.  Since $S \neq \emptyset$, it is enough to show that for all $x,y,z\in S$, we have $x+y-z \in S$. Let $P,Q,R \in X(k)$ be preimages of $x,y,z \in S$. 
        Assume that there is a plane containing \(P,Q,R\) that intersects \(X\) in a smooth genus \(1\) curve \(\Gamma\).  Then, by applying Riemann-Roch to the degree $1$ divisor $P+Q-R$ in $\Gamma$, we find $T\in \Gamma(k)$ rationally equivalent to $P+Q-R$. As rationally equivalent zero cycles are Brauer equivalent, this shows that $\ev_{\alpha}(T)=x+y-z$ and so by definition $x+y - z\in S$.
        
        It remains to consider the case that all planes containing \(P,Q,R\) intersect \(X\) in a singular curve.  Since the evaluation map is locally constant, we may perturb \(P, Q,\) or \(R\) in analytic neighborhoods and they will remain preimages of \(x,y,z\in S\), respectively.  Since \(X\) is smooth, the singular hyperplane sections are a proper closed subset of \(\dualP{3}\). Thus by perturbing \(P,Q,\) or \(R\), we may move off this closed subset and hence find a smooth hyperplane section of \(X\) that contains \(P,Q,R\), thereby reducing to the previous case. (This argument follows that of~\cite{CTP-CubicSurfaces}*{Proof of Lemma 3.4}; alternatively, one may appeal to~\cite{CT-CorayVariations}*{Prop. 2.7} and the more general Bertini arguments in Remark~\ref{rmk:Bertini}.)
    \end{proof}

\begin{lemma}\label{lem:BMObsOnCubic}
    Let $X$ be a smooth cubic surface over a global field $k$.  Assume that $X$ is everywhere locally soluble.  If $X(\A_k)^{\Br} = \emptyset$, then there exists an $\alpha\in \Br X[3]$ such that $X(\A_k)^{\alpha} = \emptyset$ and such that $\ev_{\alpha}\colon X(k_v) \to (\Br k_v)[3]$ is constant for all $v\in \Omega_k$.
\end{lemma}

\begin{proof}
    The existence of \(\alpha\in \Br X[3]\) such that \(X(\A_k)^{\alpha} = \emptyset\) follows from~\cite{CTP-CubicSurfaces}*{Lemma 3.4 and Remark 1 following the lemma}. (Note that although~\cite{CTP-CubicSurfaces}*{Lemma 3.4} is stated for number fields, the proof applies to all global fields.) By Lemma~\ref{lem:ImageEvIsCoset}, for each $v\in \Omega_k$, the image of $\ev_{\alpha}$ is a group coset.  Thus, $\ev_{\alpha}\colon X(k_v) \to (\Br k_v)[3]$ is surjective or constant (or both, if \(v\) is archimedean!).  Assume there exists a nonarchimedean $v_0\in \Omega_k$ such that $\ev_{\alpha}$ is surjective on $X(k_{v_0})$ points.  Then, for any $(P_v)\in \prod_{v\neq v_0}X(k_v)$, there exists a $P_{v_0}\in X(k_{v_0})$ such that
    \[
        \inv_{v_0}(\ev_{\alpha}(P_{v_0})) = \sum_{v\neq v_0} \inv_{v}(\ev_{\alpha}(P_{v})),
    \]
    so in particular $(P_v)\in X(\A_k)^{\alpha}$, which contradicts the first statement.  Thus, $\ev_{\alpha}\colon X(k_v) \to (\Br k_v)[3]$ is constant for all $v\in \Omega_k$.
\end{proof}

\begin{proof}[Proof of Theorem ~\ref{thm:Main}] 
    The implication $X(\A_L)^{\Br} = \emptyset \Rightarrow X(\A_k)^{\Br} = \emptyset$ follows from the compatibility of the Brauer-Manin pairing with corestriction (this appears to have been observed in this generality only recently, and is due to Wittenberg; see~\cite{CV-dp4s}*{Lemma 2.1}).  Thus, it remains to prove the reverse implication, so we assume that $X(\A_k)^{\Br} = \emptyset$.  
    
    If $X(\Adeles_L) = \emptyset$, then the result is immediate.  Assume that $X(\Adeles_L)\neq\emptyset$.  Since $3\nmid[L:k]$, for every $v\in \Omega_k$, there exists a $w\in \Omega_L$, $w|v$ such that $[L_w:k_v]$ is also coprime to $3$.  Since $X(L_w)\neq \emptyset$, by~\cite{Coray-Cubic}*{Thm. 4.7}, $X(k_v)\neq \emptyset$.  Hence $X(\A_k)\neq\emptyset$.  Since, by assumption, $X(\A_k)^{\Br} = \emptyset$,  Lemma~\ref{lem:BMObsOnCubic} implies that there exists an $\alpha\in \Br X[3]$ such that $X(\A_k)^{\alpha} = \emptyset$ and, for all $v\in\Omega_k$, $\ev_{\alpha}\colon X(k_v) \to (\Br k_v)[3]$ is constant.  Let \(c_v\in \Q/\Z\) be the image of \(\inv_v\circ \ev_{\alpha}\).  Note that by our assumptions on \(\alpha\), \(\sum_v c_v \neq 0 \in \frac13\Z/\Z\).
    
    We will show that \(X(\A_L)^{\alpha} = \emptyset\), which then implies that \(X(\A_L)^{\Br} = \emptyset\).  Let $(P_w)\in X(\A_L)$.  
    By Proposition~\ref{prop:ConstantPersistence}, $\inv_w(\ev_{\alpha_{L_w}}(P_w)) = [L_w:k_v]c_v$.  Thus,
    \[
        \sum_{w\in \Omega_L}\inv_w(\ev_{\alpha_{L_w}}(P_w)) = 
        \sum_{v\in \Omega_k}\sum_{w|v} [L_w:k_v] c_v = \sum_{v\in\Omega_k}c_v\sum_{w|v} [L_w:k_v] = [L:k]\sum_vc_v.
    \]
    Recall that \(\sum_vc_v\neq 0 \in \frac13\Z/\Z\).  Since $3\nmid [L:k]$, the above expression is also nonzero in $\frac13\Z/\Z$, and so \((P_w)\notin X(\A_L)^{\alpha}\).  Since the above argument holds for any \((P_w)\in X(\A_L)\), we have shown that \(X(\A_L)^{\alpha} = \emptyset\), as desired.
\end{proof}

\appendix
\section{Bertini for degree $1$ del Pezzo surfaces in characteristic $2$}

\begin{prop}\label{char2Bertini}
    Let $k$ be a field of characteristic $2$ and let $S$ be a smooth del Pezzo surface of degree $1$ over $k$.  Let $\phi\colon S\to Q\subset \PP^3$ be the map given by the linear system $|-2K_S|$, where $Q\subset \PP^3$ denotes the quadric cone.  Then there is a dense open $U\subset \dualP{3} $ such that, for all $H\in U$, the fiber $S_H$ is smooth.
\end{prop}
\begin{remark}
    Note that $\phi$ is a ramified double cover which is not residually separable, and so, to the best of our knowledge, no general Bertini theorems apply.
\end{remark}
\begin{proof}
    Let $R\subset S$ denote the ramification locus of $\phi$.  If $x\in S-R$, then $\phi|_H$ is smooth at $x$ for all $H$ containing $\phi(x)$.  Let us consider
    \[
        W = \left\{(r, H) : \phi(r) \in H \textup{ and }S_H \textup{ is not smooth at } r\right\} \subset R\times \dualP{3}. 
    \]
    To prove the theorem, we must show that the second projection $W\to \dualP{3}$ is not dominant, i.e., that the image has dimension at most $2$.
    
    Recall that $S$ can be given as the vanishing of a sextic hypersurface in $\PP(1,1,2,3)$, and under this identification $\phi$ is the projection onto $\PP(1,1,2)$.  Let $F$ denote the degree $6$ polynomial that defines $S$, and let the $x,y,z,w$ denote the variables of weights $1,1,2,3$, respectively.  Then $R$ is given by the vanishing of the equation $\partial_wF$, which is nonzero since $S$ is smooth.  

    We will show that over an $U\subset R$, the morphism $\pi_1 \colon W_U \to U$ has $1$ dimensional fibers.  Thus, the image of $W_U$ has dimension at most $2$.  Since $\pi_2(W_P)$ is contained in a hyperplane for any $P\in R$, this suffices to show that the map $\pi_2\colon W \to \dualP{3}$ is not dominant.
    
    Let $P\in R$ and let $H\in \A^3$.  We will restrict to considering $H$ that are given by an equation of the form $z + a_0x^2 + a_1xy + a_2y^2$. Then $S_H$ is singular at $P$ if $[a_0, a_1, a_2, b]$ is in the kernel of the matrix
    \[
    \begin{pmatrix}
        x(P)^2  & x(P)y(P) & y(P)^2 & z(P)\\
         0 & y(P)(\partial_zF)(P)&0 & (\partial_xF)(P) \\
         0 & x(P)(\partial_zF)(P)&0 & (\partial_yF)(P) 
    \end{pmatrix}.
    \]
    Over an open set $U\subset R$ we may assume that one of $x, y$ and that one of $\partial_xF, \partial_yF, \partial_z F$ are nonzero at $P$.  Thus, this matrix has rank at least $2$, and so the fiber of $\pi_1$ at $P$ is a $\PP^1$, as desired.
\end{proof}


\begin{bibdiv}
		\begin{biblist}

\bib{CTConj}{article}{
   author={Colliot-Th\'{e}l\`ene, Jean-Louis},
   title={Points rationnels sur les fibrations},
   language={French},
   conference={
      title={Higher dimensional varieties and rational points},
      address={Budapest},
      date={2001},
   },
   book={
      series={Bolyai Soc. Math. Stud.},
      volume={12},
      publisher={Springer, Berlin},
   },
   date={2003},
   pages={171--221},
   review={\MR{2011747}},
   doi={10.1007/978-3-662-05123-8\_7},
}
		    \bib{CT-CorayVariations}{article}{
   author={Colliot-Th\'{e}l\`ene, Jean-Louis},
   title={Z\'{e}ro-cycles sur les surfaces de del Pezzo (Variations sur un th\`eme
   de Daniel Coray)},
   language={French, with English and French summaries},
   journal={Enseign. Math.},
   volume={66},
   date={2020},
   number={3-4},
   pages={447--487},
   issn={0013-8584},
   review={\MR{4254269}},
   doi={10.4171/lem/66-3/4-8},
    }
    \bib{CTP-CubicSurfaces}{article}{
        author={Colliot-Th\'{e}l\`ene, Jean-Louis},
        author={Poonen, Bjorn},
        title={Algebraic families of nonzero elements of Shafarevich-Tate groups},
        journal={J. Amer. Math. Soc.},
        volume={13},
        date={2000},
        number={1},
        pages={83--99},
        issn={0894-0347},
        review={\MR{1697093}},
        doi={10.1090/S0894-0347-99-00315-X},
    }

    \bib{CTSalberger}{article}{
        author={Colliot-Th\'{e}l\`ene, Jean-Louis},
        author={Salberger, Per},
        title={Arithmetic on some singular cubic hypersurfaces},
        journal={Proc. London Math. Soc. (3)},
        volume={58},
        date={1989},
        number={3},
        pages={519--549},
        issn={0024-6115},
        review={\MR{988101}},
        doi={10.1112/plms/s3-58.3.519},
        }

        \bib{CTS-Brauer}{book}{
   author={Colliot-Th\'{e}l\`ene, Jean-Louis},
   author={Skorobogatov, Alexei N.},
   title={The Brauer-Grothendieck group},
   series={Ergebnisse der Mathematik und ihrer Grenzgebiete. 3. Folge. A
   Series of Modern Surveys in Mathematics [Results in Mathematics and
   Related Areas. 3rd Series. A Series of Modern Surveys in Mathematics]},
   volume={71},
   publisher={Springer, Cham},
   date={2021},
   pages={xv+453},
   isbn={978-3-030-74247-8},
   isbn={978-3-030-74248-5},
   review={\MR{4304038}},
   doi={10.1007/978-3-030-74248-5},
}
    \bib{Coray-Cubic}{article}{
        author={Coray, D. F.},
        title={Algebraic points on cubic hypersurfaces},
        journal={Acta Arith.},
        volume={30},
        date={1976},
        number={3},
        pages={267--296},
        issn={0065-1036},
        review={\MR{429731}},
        doi={10.4064/aa-30-3-267-296},
    }
    
    \bib{CV-dp4s}{article}{
        author={Creutz, Brendan},
        author={Viray, Bianca},
        title={Quadratic points on intersections of two quadrics},
        note={Preprint, {\tt arXiv:2106.08560}}
    }

            \bib{MO}{misc}{    
    title={Bertini theorems for base-point-free linear systems in positive characteristics},    
    author={Geraschenko, Anton},    
    note={URL: \url{https://mathoverflow.net/q/73508} (version: 2011-08-23) author:  \url{https://mathoverflow.net/users/1/anton-geraschenko}},    
    eprint={https://mathoverflow.net/q/73508},    
    organization={MathOverflow}  
}
\bib{Jouanolou}{book}{
    author={Jouanolou, Jean-Pierre},
    title={Th\'{e}or\`emes de Bertini et applications},
    language={French},
    series={Progress in Mathematics},
    volume={42},
    publisher={Birkh\"{a}user Boston, Inc., Boston, MA},
    date={1983},
    pages={ii+127},
    isbn={0-8176-3164-X},
    review={\MR{725671}},
 }
\bib{Poonen-RationalPoints}{book}{
    author={Poonen, Bjorn},
    title={Rational points on varieties},
    series={Graduate Studies in Mathematics},
    volume={186},
    publisher={American Mathematical Society, Providence, RI},
    date={2017},
    pages={xv+337},
    isbn={978-1-4704-3773-2},
    review={\MR{3729254}},
    doi={10.1090/gsm/186},
 }
 
            \bib{Spreafico}{article}{
   author={Spreafico, Maria Luisa},
   title={Axiomatic theory for transversality and Bertini type theorems},
   journal={Arch. Math. (Basel)},
   volume={70},
   date={1998},
   number={5},
   pages={407--424},
   issn={0003-889X},
   review={\MR{1612610}},
   doi={10.1007/s000130050213},
}

\bib{SwinnertonDyer62}{article}{
   author={Swinnerton-Dyer, H. P. F.},
   title={Two special cubic surfaces},
   journal={Mathematika},
   volume={9},
   date={1962},
   pages={54--56},
   issn={0025-5793},
   review={\MR{139989}},
   doi={10.1112/S0025579300003090},
}

\bib{SD-BrauerCubic}{article}{
   author={Swinnerton-Dyer, Peter},
   title={The Brauer group of cubic surfaces},
   journal={Math. Proc. Cambridge Philos. Soc.},
   volume={113},
   date={1993},
   number={3},
   pages={449--460},
   issn={0305-0041},
   review={\MR{1207510}},
   doi={10.1017/S0305004100076106},
}
		\end{biblist}
	\end{bibdiv}
\end{document}